\documentclass{article}
\usepackage{amsfonts}
\usepackage[english]{babel}
\usepackage{latexsym}
\usepackage[latin1]{inputenc}
\usepackage{amssymb}
\usepackage{amsmath}
\usepackage{bbm}
\usepackage{amsthm}
\usepackage{graphicx}

\newtheorem{thm}{Theorem}
\newtheorem{lem}[thm]{Lemma}
\newtheorem{defi}[thm]{Definition}
\newtheorem{cor}[thm]{Corolary}
\newcommand{\pd}[2]{\frac{\partial #1}{\partial #2}}
\newtheorem{nota}[thm]{Remark}

\begin{document}

\title{On a nonlocal degenerate parabolic problem}

%  Full author information.

\author{
Rui M.P. Almeida \thanks{Departamento de Matem\'{a}tica, Faculdade de Ci\^{e}ncias, Universidade da Beira Interior, e-mail: ralmeida@ubi.pt web: http://www.mat.ubi.pt/\~{}ralmeida} \and
Stanislav N. Antontsev \thanks{Centro de Matem\'{a}tica e Aplica\c{c}\~{o}es Fundamentais, Faculdade de Ci\^{e}ncias, Universidade de Lisboa, e-mail:anton@ptmat.fc.ul.pt , web: http://cmaf.ptmat.fc.ul.pt}\and
Jos\'{e} C.M. Duque \thanks{Departamento de Matem\'{a}tica, Faculdade de Ci\^{e}ncias, Universidade da Beira Interior, e-mail: jduque@ubi.pt web: http://www.mat.ubi.pt/\~{}jduque} }

\date{\today}
\maketitle

\begin{abstract}
Conditions for the existence and uniqueness of weak solutions for a class of nonlinear nonlocal degenerate parabolic equations are established. The asymptotic behaviour of the solutions as time tends to infinity are also studied. In particular, the finite time extinction and polynomial decay properties are proved.

\textbf{keywords}: nonlocal, degenerate, parabolic, PDE.
\end{abstract}

\section{Introduction}

In this paper, we study parabolic problems with nonlocal nonlinearity of the
following type:
\begin{equation}  \label{prob}
\left\{
\begin{array}{l}
\displaystyle u_t-\left( \int_{\Omega}u^2(x,t)dx\right)^{\gamma} \Delta u=f\left(x,t\right)\,, \quad (x,t)\in \Omega\times ]0,T]\\
\displaystyle u\left(x,t\right) =0\,,\quad (x,t)\in \partial\Omega\times ]0,T] \\
\displaystyle u(x,0)=u_{0}(x)\,, \quad x\in \Omega \\
\end{array}
\right. \,
\end{equation}
where $\Omega$ is a bounded open domain in $\mathbbm{R}^d$, $d\geq  1$, $\gamma$ is a real constant, $f$ and $u_0$ are continuous integrable functions.\\
This type of problem was studied initially by Chipot and Lovat in \cite{CL97}, where they proposed the equation
\begin{equation}\label{eq_cl}
u_t-a(\int_{\Omega} u \ dx)\Delta u =f
\end{equation}
for modelling the density of a population, for example, of bacteria, subject to spreading. This equation also appears in the study of heat propagation or in epidemic theory. In this paper the authors prove the existence and uniqueness of a weak solution to this equation.\\
In \cite{CL97}, the authors studied the problem
\begin{equation}\label{eqcl}
\left\{
\begin{array}{lrl}
u_{t}-a(l(u))\Delta u=f(x,t)&\text{ in }&\Omega \times (0,T) \\
u(x,t)=0&\text{ on }&\partial \Omega \times (0,T) \\
u(x,0)=u_{0}(x)&\text{ in }&\Omega %
\end{array}%
\right.
\end{equation}%
where $\Omega $ is a bounded open subset in $\mathbbm{R}^{d}$, $d\geq 1$, with smooth
boundary $\partial \Omega $, $T$ is some arbitrary time and $a$ is some
function from $\mathbbm{R}$ into $(0,+\infty )$. Both $a$ and $f$ are
continuous functions and $l:L_{2}(\Omega )\rightarrow \mathbbm{R}$ is a continuous
linear form.
The existence, uniqueness and asymptotic behaviour of weak and strong solutions of parabolic equations and systems with nonlocal diffusion terms have been widely studied in the last two decades.\\
In 2000, Ackleh and Ke \cite{AK00}, studied the problem
$$\left\{\begin{array}{lrl}
u_t=\frac{1}{a(\int_{\Omega}u\ dx)}\Delta u+f(u)&\text{ in }&\Omega\times ]0,T]\\
u(x,t)=0&\text{ on }&\partial \Omega\times ]0,T]\\
u(x,0)=u_0(x)&\text{ in }&\overline \Omega
\end{array}\right.,$$
with $a(\xi)>0$ for all $\xi\neq0$, $a(0)\geq 0$ and $f$ Lipschitz-continuous satisfying $f(0)=0$.
They proved the existence and uniqueness of a solution to this problem and gave conditions on $u_0$ for the extinction in finite time and for the persistence of solutions.
The asymptotic behaviour of the solutions as time tends to infinity was studied
by Zheng and Chipot \cite{ZC05} for nonlinear parabolic equations with two
classes of nonlocal terms, in a cylindrical domain.
Recently, Duque et al. \cite{DAAFppa} considered a nonlinear coupled system of reaction-diffusion on
a bounded domain with a more general nonlocal diffusion term working on two
linear forms $l_1$ and $l_2$:
\begin{equation}\label{sys0}
\left\{\begin{array}{lrl}
u_t-a_1(l_1(u),l_2(v))\Delta u+\lambda_1|u|^{p-2}u=f_1(x,t)&\text{ in }&\Omega\times ]0,T]\\
v_t-a_2(l_1(u),l_2(v))\Delta v+\lambda_2|v|^{p-2}v=f_2(x,t)&\text{ in }& \Omega\times ]0,T]\\
\end{array}\right..
\end{equation}
In this case, $u$ and $v$ could describe the
densities of two populations that interact through the functions $a_1$ and $a_2$. The death in species $u$ is assumed to be proportional to $|u|^{p-2}u$ by the factor $\lambda_1\geq 0$ and in species $v$ to be proportional to $|v|^{p-2}v$ by the factor $\lambda_2\geq 0$ with $p\geq1$. The supply of being by external sources is denoted by $f_1$ and $f_2$. The authors proved the existence and uniqueness of weak and strong global in time solutions and imposed conditions, on the data, for these solutions to have the waiting time and stable localization properties. Moreover, important results on polynomial and exponential decay and vanishing of the solutions in finite time were also presented.\\
Robalo et al. \cite{RACFppa} proved the existence and uniqueness of weak and
strong global in time solutions and gave conditions, on the data, for these
solutions to have the exponential decay property for a nonlocal problem with moving boundaries.\\
The numerical analysis and simulation of such problems were less studied (see, for example, \cite{BS09,DAAFppb,ADFRpp,ADFRppb,RACFppb} and their references).\\
In this work, we analyse a different diffusion term, dependent on the $L_2$-norm of the solution. In most of the previous papers, it is assumed that the diffusion term is bounded, with $0<m\leq a(s)\leq M<\infty$, $s\in\mathbbm{R}$, so the problem is always nondegenerate. Here, we study a case were the diffusion term could be zero or infinity. This work is concerned with the proof of the existence, uniqueness and asymptotic behaviour of the weak solutions. To the best of our knowledge, for nonlocal reaction-diffusion equations with this type of diffusion term, these results have not yet been established.\\
The paper is organized as follows. In Section 2, we formulate the problem and the hypotheses on the
data. In Section 3, we define an auxiliary problem and prove the existence of weak solutions for the initial problem. Section 4 is devoted to the proof of the uniqueness of weak solutions. In Section 5, we study the asymptotic behaviour of the weak solution. Finally, in Section 6, we draw some conclusions.

\section{Statement of the problem}
Let us consider the problem of finding the function $u$ which satisfies the
following conditions:
\begin{equation}  \label{probi}
\left\{
\begin{array}{l}
\displaystyle \frac{\partial u}{\partial t}- a(u)\Delta u=f\left(x,t\right)\,, \quad (x,t)\in \Omega\times ]0,T]\\
\displaystyle u\left(x,t\right) =0\,,\quad (x,t)\in \partial\Omega\times ]0,T] \\
\displaystyle u(x,0)=u_{0}(x)\,, \quad x\in \Omega \\
\end{array}
\right. \,
\end{equation}
where $\Omega$ is a bounded open domain in $\mathbbm{R}^d$, $d\geq 1$, $a(u)= \left(\int_{\Omega}u^2(x,t)dx\right)^{\gamma}$ with $\gamma\in\mathbbm{R}$, and $f$ and $u_0$ are continuous integrable functions.\\
If $\gamma=0$, we have the heat equation which is widely known. For $\gamma>0$, the problem could degenerate if there is an extinction phenomenon, and for $\gamma<0$, if the extinction occurs, the problem becomes singular.\\
The definition of a weak solution to this problem is as follows:
\begin{defi}[Weak solution]\label{fraca}
We say that the function $u$ is
a weak solution of problem (\ref{probi}) if
\begin{equation}
u\in L_{2}(0,T;H^{1}(\Omega)),\frac{\partial u}{\partial t}\in L_{2}(0,T;L_{2}(\Omega)),  \label{regx}
\end{equation}
the equality
\begin{equation}\label{fracav}
\int_{\Omega}u_tw\ dx+a(u)\int_{\Omega}\nabla u\cdot\nabla w\ dx=\int_{\Omega}fw\ dx
\end{equation}
is valid for all $w\in H_{0}^{1}(\Omega)$ and $t\in ]0,T[$, and
\begin{equation}\label{condi}
u(x,0)=u_{0}(x),\quad x\in \Omega.
\end{equation}
\end{defi}

\section{Existence of a weak solution}
Since the problem may be degenerate, we consider the auxiliary problem
\begin{equation}  \label{prob_aux}
\left\{
\begin{array}{l}
\displaystyle \frac{\partial u_\varepsilon}{\partial t}-a_\varepsilon(u_\varepsilon) \Delta u_\varepsilon=f\left(x,t\right)\,, \quad (x,t)\in \Omega\times ]0,T]\\
\displaystyle u_\varepsilon\left(x,t\right) =0\,,\quad (x,t)\in \partial\Omega\times ]0,T] \\
\displaystyle u_\varepsilon(x,0)=u_{0}(x)\,, \quad x\in \Omega \\
\end{array}
\right. \,
\end{equation}
where
%$$a_{\varepsilon}(u_{\varepsilon})=\frac{\int_\Omega u_\varepsilon^2\ dx +\varepsilon}{\varepsilon\int_\Omega u_\varepsilon^2\ dx +1},\quad\varepsilon\in]0,1].$$
$$a_{\varepsilon}(u_{\varepsilon})=\left(\min\{a(u_{\varepsilon}),K^2\}+\varepsilon\right)^{\gamma},\quad\varepsilon\in]0,1],$$
with $K$ a finite parameter to be chosen later.

Since $0< m(\gamma,\varepsilon,K)\leq a_\varepsilon\leq M(\gamma,\varepsilon,K)<\infty$, for every $\varepsilon>0$, Problem (\ref{prob_aux}) has a unique strong solution $u_\varepsilon$, which satisfies $\frac{\partial u_\varepsilon}{\partial t}$, $\Delta u_\varepsilon \in L_2(\Omega)$ (see \cite{LSU68}). First, we define the parameter $K$ and next, we obtain estimates, independent of $\varepsilon$, for $u_\varepsilon$ and its derivatives.

\begin{lem}\label{maj2_u}
Let $u_\varepsilon$ be a weak solution of Problem (\ref{prob_aux}). If $u_0\in L_{2k}(\Omega)$ and $f\in L_1(0,T;L_{2k}(\Omega))$, for $k\in\mathbbm{N}$, then
\begin{equation}\label{est_u}
\|u_{\varepsilon}\|_{L_{2k}(\Omega)}\leq \|u_0\|_{L_{2k}(\Omega)}+\int_0^T\| f\|_{L_{2k}(\Omega)}\ dt\leq C,
\end{equation}
where $C$ does not depend on $\varepsilon$.
\end{lem}

\begin{proof}
Multiplying the first equation of Problem (\ref{prob_aux}) by $u_\varepsilon^{2k-1}$ and integrating in $\Omega$, we arrive at
$$\frac{1}{2k}\frac{d}{dt}\| u_\varepsilon\|_{L_{2k}(\Omega)}^{2k}+(2k-1)a_\varepsilon(u_\varepsilon)\int_{\Omega}u_\varepsilon^{2k-2}|\nabla u_\varepsilon|^2\ dx =\int_{\Omega} fu_\varepsilon^{2k-1}\ dx.$$
Applying the H\"older inequality and ignoring the second term ( since it is non-negative ) on the left hand side, we obtain
$$\frac{1}{2k}\frac{d}{dt}\| u_\varepsilon\|_{L_{2k}(\Omega)}^{2k} \leq \| f\|_{L_{2k}(\Omega)} \|u_\varepsilon\|_{L_{2k}(\Omega)}^{2k-1}.$$
Simplifying the factor $\|u_\varepsilon\|_{L_{2k}(\Omega)}^{2k-1}$ and integrating in $t$, we have (\ref{est_u}).
\end{proof}

%\begin{lem}\label{maj1_u}
%Let $u_\varepsilon$ be a weak solution of problem (\ref{prob_aux}). If
%$$\left(\int_{\Omega} u_0^2\ dx\right)^{\frac12}+\int_0^T\int_{\Omega} f^2\ dxdt\leq K,$$
%then
%$$\left(\int_{\Omega} u_\varepsilon^2\ dx\right)^{\frac12}\leq K.$$

%\end{lem}
%\begin{proof}
%Multiplying the first equation of Problem (\ref{prob_aux}) by $u_\varepsilon$ and integrating in $\Omega$, we arrive at
%$$\frac12\frac{d}{dt}\int_{\Omega} u_\varepsilon^2\ dx+a_\varepsilon(u_\varepsilon)\int_{\Omega}|\nabla u_\varepsilon|^2\ dx =\int_{\Omega} fu_\varepsilon\ dx.$$
%Then
%$$\left(\int_{\Omega} u_\varepsilon^2\ dx\right)^{\frac12}\leq \left(\int_{\Omega} u_0^2\ dx\right)^{\frac12}+\int_0^T\int_{\Omega} f^2\ dxdt,$$
%and the result follows.
%\end{proof}

\begin{cor}\label{maj_a}
Let $u_\varepsilon$ be a weak solution of Problem (\ref{probi}). If the conditions of Lemma \ref{maj2_u} are fulfilled for $k=2$, then we have
\begin{equation}\label{Ma}
a_{\varepsilon}(u_{\varepsilon})=\left(\int_\Omega u_\varepsilon^2\ dx +\varepsilon\right)^{\gamma}
\end{equation}
and
$$0<(K^2+1)^\gamma\leq a_{\varepsilon}(u_{\varepsilon})\leq\varepsilon^\gamma<\infty,\text{ if }\gamma<0,$$
$$0<\varepsilon^\gamma\leq a_{\varepsilon}(u_{\varepsilon})\leq(K^2+1)^\gamma<\infty,\text{ if }\gamma>0,$$
with $$K=\|u_0\|_{L_{2}(\Omega)}+\int_0^T\| f\|_{L_{2}(\Omega)}\ dt<\infty.$$
\end{cor}

\begin{lem}\label{maj_ux}
Let $u_\varepsilon$ be a weak solution of Problem (\ref{prob_aux}). If $u_0\in H_0^1(\Omega)$ and $f\in L_2(0,T;H_0^1(\Omega))$,
then\\

$\displaystyle \int_{\Omega}|\nabla u_\varepsilon|^2\ dx +\int_0^Ta_\varepsilon\int_{\Omega}(\Delta u_\varepsilon)^2\ dxdt$
\begin{equation}\label{est_ux}
\leq C\int_{\Omega}|\nabla u_0|^2\ dx+C\int_0^T\int_{\Omega} |\nabla f|^2\ dxdt\leq C,
\end{equation}
where $C$ does not depend on $\varepsilon$.
\end{lem}
\begin{proof}
Multiplying the same equation as before by $\Delta u_\varepsilon$ and integrating in $\Omega$, we conclude that
$$\frac12\frac{d}{dt}\int_{\Omega}|\nabla u_\varepsilon|^2\ dx+a_\varepsilon\int_{\Omega}(\Delta u_\varepsilon)^2\ dx=-\int_{\Omega} f\Delta u_\varepsilon\ dx.$$
If we apply Green's theorem to the right hand side, the last equation becomes
$$\frac12\frac{d}{dt}\int_{\Omega}|\nabla u_\varepsilon|^2\ dx+a_\varepsilon\int_{\Omega}(\Delta u_\varepsilon)^2\ dx=\int_{\Omega} \nabla f\cdot \nabla u_\varepsilon\ dx.$$
Using the H\"older and Cauchy inequalities, we obtain
$$\frac12\frac{d}{dt}\int_{\Omega}|\nabla u_\varepsilon|^2\ dx+a_\varepsilon\int_{\Omega}(\Delta u_\varepsilon)^2\ dx\leq \frac12\int_{\Omega} |\nabla u_\varepsilon|^2\ dx+\frac12\int_{\Omega} |\nabla f|^2\ dx.$$
Hence, applying Gronwall's lemma, we conclude that (\ref{est_ux}) is true.
%$$\int_{\Omega}|\nabla u_\varepsilon|^2\ dx\leq C\int_{\Omega}|\nabla u_0|^2\ dx+C\int_0^t\int_{\Omega} |\nabla f|^2\ dxdt.$$
\end{proof}

\begin{cor}\label{maj_uxx1}
If the conditions of Lemma \ref{maj_ux} are fulfilled and $\gamma<0$, then
$$\int_0^T\int_{\Omega}(\Delta u_\varepsilon)^2\ dxdt\leq C,$$
with $u_\varepsilon$ a weak solution of Problem (\ref{prob_aux}).
\end{cor}

\begin{cor}\label{maj_uxx1}
If the conditions of Lemma \ref{maj_ux} are fulfilled and $\gamma>0$, then
$$\int_0^T\int_{\Omega}(\Delta u_\varepsilon)^2\ dxdt\leq C\varepsilon^{-\gamma},$$
with $u_\varepsilon$ a weak solution of Problem (\ref{prob_aux}).
\end{cor}

To prove the bounds for the temporal derivative, we will consider the two different cases:  $\gamma\geq0$ and  $\gamma<0$.

\begin{lem}\label{maj_ut}
Let $u_\varepsilon$ be a weak solution of Problem (\ref{prob_aux}) with $\gamma\geq 0$. If $u_0\in H_0^1(\Omega)$ and $f\in L_2(0,T;H_0^1(\Omega))$, then\\

$\displaystyle \int_0^T\int_{\Omega}\left(\frac{\partial u_\varepsilon}{\partial t}\right)^2\ dxdt$
\begin{equation}\label{est_ut}
\leq C\left(\int_{\Omega}|\nabla u_0|^2\ dx+\int_0^T\int_{\Omega} |\nabla f|^2\ dxdt +\int_0^T\int_{\Omega} f^2\ dxdt\right),
\end{equation}
where $C$ does not depend on $\varepsilon$.
\end{lem}
\begin{proof}
Multiplying the first equation of Problem (\ref{prob_aux}) by $\frac{\partial u_\varepsilon}{\partial t}$ and integrating in $\Omega\times[0,T]$, we arrive at the inequality
$$\int_0^T\int_{\Omega}\left(\frac{\partial u_\varepsilon}{\partial t}\right)^2\ dxdt=\int_0^Ta_\varepsilon\int_{\Omega}\Delta u_\varepsilon \frac{\partial u_\varepsilon}{\partial t}\ dxdt + \int_0^T\int_{\Omega} f\frac{\partial u_\varepsilon}{\partial t}\ dxdt.$$
Applying Cauchy's inequality to the right hand side, we obtain
$$\int_0^T\int_{\Omega}\left(\frac{\partial u_\varepsilon}{\partial t}\right)^2\ dxdt\leq\frac14\int_0^T\int_{\Omega}\left(\frac{\partial u_\varepsilon}{\partial t}\right)^2\ dxdt+\int_0^Ta_\varepsilon^2\int_{\Omega}(\Delta u_\varepsilon)^2\ dxdt$$
$$+\frac14\int_0^T\int_{\Omega}\left(\frac{\partial u_\varepsilon}{\partial t}\right)^2\ dxdt+\int_0^T\int_{\Omega} f^2\ dxdt,$$
that is,
$$\frac12\int_0^T\int_{\Omega}\left(\frac{\partial u_\varepsilon}{\partial t}\right)^2\ dxdt\leq(K^2+1)^{2\gamma}\int_0^Ta_\varepsilon\int_{\Omega}(\Delta u_\varepsilon)^2\ dxdt+\int_0^T\int_{\Omega} f^2\ dxdt.$$
Using Lemma \ref{maj_ux}, we obtain
$$\int_0^T\int_{\Omega}\left(\frac{\partial u_\varepsilon}{\partial t}\right)^2\ dxdt\leq C\int_{\Omega}|\nabla u_0|^2\ dx+C\int_0^T\int_{\Omega} |\nabla f|^2\ dxdt +2\int_0^T\int_{\Omega} f^2\ dxdt$$
and the result follows.
\end{proof}

\begin{lem}\label{lema_t*2}
If $u_0\in H_0^1(\Omega)$, $f\in L_2(0,T;L_2(\Omega))$, $\int_{\Omega} u_0\ dx>0$ and $\gamma< 0$, then there exists a $t^*>0$ such that $a_{\varepsilon}(u_{\varepsilon})\leq M<\infty$ for $t\in[0,t^*]$ with $u_{\varepsilon}$ a weak solution of Problem (\ref{prob_aux}).
\end{lem}
\begin{proof}
Multiplying the first equation of Problem (\ref{prob_aux}) by $ \frac{\partial u_\varepsilon}{\partial t}$ and integrating in $\Omega$, we arrive at the equality
$$\int_{\Omega}\left(\frac{\partial u_\varepsilon}{\partial t}\right)^2\ dx-a_\varepsilon\int_{\Omega}\Delta u_\varepsilon \frac{\partial u_\varepsilon}{\partial t}\ dx = \int_{\Omega} f\frac{\partial u_\varepsilon}{\partial t}\ dx.$$
Applying Green's theorem to the second term on the left hand side and Cauchy's inequality to the right hand side, we obtain
$$\frac{1}{a_\varepsilon}\int_{\Omega}\left(\frac{\partial u_\varepsilon}{\partial t}\right)^2\ dx+\frac{d}{dt}\int_{\Omega}|\nabla u_\varepsilon|^2\ dx\leq \frac{1}{a_\varepsilon}\int_{\Omega} f^2\ dx.$$
Integrating in $t$ and using the bounds of $a$,  we conclude that\\

$\displaystyle \int_0^T\left(\int_{\Omega} u_\varepsilon^2\ dx\right)^{|\gamma|}\int_{\Omega}\left(\frac{\partial u_\varepsilon}{\partial t}\right)^2\ dxdt+\int_{\Omega}|\nabla u_\varepsilon|^2\ dx$\\

\begin{equation}\label{eq_aux}\hfill\displaystyle\leq \int_{\Omega}|\nabla u_0|^2\ dx+C\int_0^T\int_{\Omega} f^2\ dxdt.\end{equation}
Now we consider the functional
$$b(t)=\left(\int_{\Omega} u_\varepsilon^2\ dx\right)^{\beta},$$
with $\beta=\frac{|\gamma|+1}{2}$. Differentiating $b$, we obtain\\

$\displaystyle |b'(t)|=2\beta \left|\int_{\Omega} u_\varepsilon\frac{\partial u_\varepsilon}{\partial t}\ dx\right|\left(\int_{\Omega} u_\varepsilon^2\ dx\right)^{\beta-1}$\\

$\displaystyle \phantom{|b'(t)|}\leq 2\beta \left(\int_{\Omega} \left(\frac{\partial u_\varepsilon}{\partial t}\right)^2\ dx\right)^{\frac12}\left(\int_{\Omega} u_\varepsilon^2\ dx\right)^{\beta-\frac12}.$\\
Thus, using (\ref{eq_aux}) we conclude that
$$\int_0^T|b'(t)|^2\ dt\leq 4\beta^2 \int_0^T\int_{\Omega} \left(\frac{\partial u_\varepsilon}{\partial t}\right)^2\ dx\left(\int_{\Omega} u_\varepsilon^2\ dx\right)^{|\gamma|}dt\leq C_1,$$
where $C_1$ depends on $\int_{\Omega}|\nabla u_0|^2\ dx$ and $\int_0^T\int_{\Omega} f^2\ dxdt$.\\
On the other hand,
$$b(t)=b(0)+\int_0^tb'(t)\ dt\geq b(0)-\int_0^T|b'(t)|\ dt\geq b(0)-t^{\frac12}\left(\int_0^T|b'(t)|^2\ dt\right)^{\frac12}.$$
Hence, considering $t^*=\frac{b^2(0)}{C_1}$, we have that $b(t)\geq C>0$, for all $t\in[0,t^*]$.\\
Finally,
$$a_{\varepsilon}(u_{\varepsilon})\leq  \frac{1}{\left(\int_{\Omega} u_{\varepsilon}^2\ dx\right)^{|\gamma|}}=\frac{\int_{\Omega}u_{\varepsilon}^2\ dx}{b^2}\leq M<\infty,\text{ for all }t\in[0,t^*].$$
\end{proof}

\begin{lem}\label{maj_ut2}
Let $u_\varepsilon$ be a weak solution of Problem (\ref{prob_aux}) with $\gamma<0$. If $u_0\in H_0^1(\Omega)$ and $f\in L_2(0,T;L_2(\Omega))$, then there exists a $t^*>0$ such that, for $T\leq t^*$,
\begin{equation}\label{est_ut2}
\int_0^T\int_{\Omega}\left(\frac{\partial u_\varepsilon}{\partial t}\right)^2\ dxdt+\int_{\Omega}|\nabla u_\varepsilon|^2\ dx\leq C\int_{\Omega}|\nabla u_0|^2\ dx+C\int_0^T\int_{\Omega} f^2\ dxdt,
\end{equation}
where $C$ does not depend on $\varepsilon$, but may depend on $M$ and $K$.
\end{lem}
\begin{proof}
Multiplying the first equation of Problem (\ref{prob_aux}) by $ \frac{\partial u_\varepsilon}{\partial t}$ and integrating in $\Omega$, we arrive at the equality
$$\int_{\Omega}\left(\frac{\partial u_\varepsilon}{\partial t}\right)^2\ dx-a_\varepsilon\int_{\Omega}\Delta u_\varepsilon \frac{\partial u_\varepsilon}{\partial t}\ dx = \int_{\Omega} f\frac{\partial u_\varepsilon}{\partial t}\ dx.$$
Applying Green's theorem to the second term on the left hand side and Cauchy's inequality to the right hand side, we obtain
$$\frac12\int_{\Omega}\left(\frac{\partial u_\varepsilon}{\partial t}\right)^2\ dx+\frac{a_\varepsilon}{2}\frac{d}{dt}\int_{\Omega}|\nabla u_\varepsilon|^2\ dx\leq \frac12\int_{\Omega} f^2\ dx,$$
which is
$$\frac{1}{a_\varepsilon}\int_{\Omega}\left(\frac{\partial u_\varepsilon}{\partial t}\right)^2\ dx+\frac{d}{dt}\int_{\Omega}|\nabla u_\varepsilon|^2\ dx\leq \frac{1}{a_\varepsilon}\int_{\Omega} f^2\ dx.$$
Integrating in $t$ and using the bounds of $a$,  we conclude that
$$\int_0^T\int_{\Omega}\left(\frac{\partial u_\varepsilon}{\partial t}\right)^2\ dxdt+\int_{\Omega}|\nabla u_\varepsilon|^2\ dx\leq C\int_{\Omega}|\nabla u_0|^2\ dx+C\int_0^T\int_{\Omega} f^2\ dxdt,$$
for $T<t^*$.
\end{proof}

\begin{thm}\label{exist1}
If $\gamma\geq 0$, $u_0\in H_0^1(\Omega)$ and $f\in L_2(0,T;H_0^1(\Omega))$, then Problem (\ref{probi}) has a weak solution, in the sense of Definition \ref{fraca}.
\end{thm}
\begin{proof}
By Lemmas \ref{maj2_u}, \ref{maj_ux} and \ref{maj_ut}, we can conclude that there exists a function $u$ and subsequences such that
\begin{itemize}
\item[] $u_\varepsilon\to u$ weakly in $L_2(\Omega\times]0,T])$,
\item[] $\nabla u_\varepsilon\to \nabla u$ weakly in $L_2(\Omega\times]0,T])$,
\item[] $ \frac{\partial u_\varepsilon}{\partial t}\to \frac{\partial u}{\partial t}$ weakly in $L_2(\Omega\times]0,T])$,
\item[] $u_\varepsilon\to u$ a.e. in $\Omega\times]0,T]$.
\end{itemize}
Since $a$ is continuous, we have that
$$a_\varepsilon(u_\varepsilon)\to a(u)\text{ a.e. in } \Omega\times]0,T].$$
Passing to the limit in
$$\int_{\Omega}\frac{\partial u_\varepsilon}{\partial t}w+a_\varepsilon(u_\varepsilon)\nabla u_\varepsilon\nabla w-fw\ dx=0,$$
we obtain
$$\int_{\Omega}\frac{\partial u}{\partial t}w+a(u)\nabla u\nabla w-fw\ dx=0.$$
Hence $u$ is a weak solution of Problem (\ref{probi}).
\end{proof}

\begin{thm}
If $\gamma<0$, $u_0\in H_0^1(\Omega)$ and $f\in L_2(0,T;H_0^1(\Omega))$, then there exists a $t^*$ such that Problem (\ref{probi}) has a weak solution for $T\leq t^*$, in the sense of Definition \ref{fraca}.
\end{thm}
\begin{proof}
By Lemmas \ref{maj2_u}, \ref{maj_ux}, \ref{maj_ut2} and using the arguments of Theorem \ref{exist1}. the result follows easily.
\end{proof}

\section{Uniqueness of a weak solution}
In order to prove the uniqueness of a weak solution of Problem (\ref{probi}), we need to obtain positive real bounds for the diffusive term, and the Lipschitz-continuity in $u$.
\begin{lem}\label{lema_t*}
If $\gamma\geq0$,  $\int_{\Omega} u_0\ dx>0$ and the conditions of Lemma \ref{maj_ut} are fulfilled, then there exists a $t^*>0$ such that $a(u)\geq m>0$ for $t\in[0,t^*]$, with $u$ a weak solution of Problem (\ref{probi}).
\end{lem}
\begin{proof}
By Lemma \ref{maj_ut}, we have that
\begin{eqnarray*}
\int_{\Omega}u\ dx&=&\int_{\Omega}u_0\ dx+\int_0^t\frac{d}{dt}\int_{\Omega} u\ dxdt=\int_{\Omega} u_0\ dx+\int_0^t\int_{\Omega}\frac{\partial u}{\partial t}\ dxdt\\
&\geq&\int_{\Omega}u_0\ dx-t^{\frac12}|\Omega|^{\frac12}\left(\int_0^t\int_{\Omega}\left(\frac{\partial u}{\partial t}\right)^2\ dxdt\right)^{\frac12}\\
&\geq&\int_{\Omega}u_0\ dx-t^{\frac12}|\Omega|^{\frac12}C,\\
\end{eqnarray*}
with $C$ depending on $\int_{\Omega}|\nabla u_0|\ dx$, $\int_0^t\int_{\Omega} f^2\ dxdt$ and $\int_0^t\int_{\Omega} |\nabla f|^2\ dxdt$.
Thus, if $t^*$ satisfies
$$t^*<\frac{\left(\int_{\Omega} u_0\ dx\right)^2}{C|\Omega|},$$
then the solution of Problem (\ref{probi}) is such that
$$\int_{\Omega}u\ dx> 0.$$
So,
$$\left(\int_{\Omega}u^2\ dx\right)^{\gamma}> 0,\quad t<t^*.$$
\end{proof}

Before proving the uniqueness of the weak solution, we present the following lemma, which proves the Lipschitz-continuity of the diffusion term.
\begin{lem}\label{Lip_a}
If
$$0<m \leq \int_{\Omega} v^2\ dx, \int_{\Omega} w^2\ dx\leq M<\infty,$$
then
$$|a(v)-a(w)|\leq C \|v-w\|,$$
where $C$ may depend on $\gamma$ , $m$ and $M$.
\end{lem}
\begin{proof}
Denoting $p=\int_{\Omega} v^2\ dx$ and $q=\int_{\Omega} w^2\ dx$, we have
$$p^{\gamma}-q^{\gamma}=\gamma\int^p_q \tau^{\gamma-1}\ d\tau.$$
If $\gamma\geq 1$, then $\tau^{\gamma-1}\leq M^{\gamma-1}$ and $|p^{\gamma}-q^{\gamma}|\leq\gamma M^{\gamma-1}|p-q|.$\\
If $\gamma<1$, then $\tau^{\gamma-1}\leq m^{\gamma-1}$ and $|p^{\gamma}-q^{\gamma}|\leq|\gamma| m^{\gamma-1}|p-q|,$ where
\begin{eqnarray*}
|p-q|&\leq& \int_{\Omega}|v^2-w^2|\ dx=\int_{\Omega}|v-w|\ |v+w|\ dx\\
 &\leq& \left(\int_{\Omega}|v-w|^2\ dx\right)^{\frac12}\left(\int_{\Omega}|v+w|^2\ dx\right)^{\frac12}\\
 &\leq& C\left(\int_{\Omega}|v-w|^2\ dx\right)^{\frac12}
\end{eqnarray*}
and this completes the proof.
\end{proof}

\begin{thm}
If $u_0\in H_0^1(\Omega)$, $f\in L_2(0,T;H_0^1(\Omega))$ and $\int_{\Omega} u_0\ dx>0$, then there exists a $t^*>0$ such that Problem (\ref{probi}) has a unique weak solution, in the sense of Definition \ref{fraca}, for $t\leq t^*$.
\end{thm}

\begin{proof}
Suppose that there exist two solutions $u_1$ and $u_2$. Let $u=u_1-u_2$, then
\begin{equation*}
\int_{\Omega}\frac{\partial u}{\partial t}w\ dx+\int_{\Omega}(a(u_1)\nabla u_1-a(u_2)\nabla u_2)\cdot\nabla w\ dx=0
\end{equation*}
or, equivalently,
\begin{equation*}
\int_{\Omega}\frac{\partial u}{\partial t}w\ dx+\int_{\Omega}a(u_1)\nabla u\cdot\nabla w\ dx=\int_{\Omega}(a(u_2)-a(u_1))\nabla u_2\cdot\nabla w\ dx.
\end{equation*}
Making $w=u$, we obtain
\begin{equation*}
\frac12\frac{d}{dt}\int_{\Omega}u^2\ dx+a(u_1)\int_{\Omega}|\nabla u|^2\ dx=(a(u_2)-a(u_1))\int_{\Omega}\nabla u_2\cdot\nabla u\ dx.
\end{equation*}
Then
\begin{equation*}
\frac12\frac{d}{dt}\int_{\Omega}u^2\ dx+a(u_1)\int_{\Omega}|\nabla u|^2\ dx\leq |a(u_2)-a(u_1)|\int_{\Omega}|\nabla u_2\cdot\nabla u|\ dx.
\end{equation*}
By Lemma \ref{lema_t*} and Corollary \ref{maj_a}, there exists $t^*>0$ such that $a(u_1)\geq m>0$, for $t\in[0,t^*]$. So we have that
\begin{equation*}
\frac{d}{dt}\int_{\Omega}u^2\ dx+m\int_{\Omega}|\nabla u|^2\ dx\leq \frac{1}{4m}|a(u_2)-a(u_1)|^2\int_{\Omega}|\nabla u_2|^2\ dx+m\int_{\Omega}|\nabla u|^2\ dx.
\end{equation*}
In the last section, we proved that
$$\int_{\Omega}|\nabla u_2|^2\ dx\leq C,$$
and, in Lemma \ref{Lip_a}, we proved that
$$|a(u_2)-a(u_1)|^2\leq \|u_2-u_1\|^2=\int_{\Omega} u^2\ dx.$$
Thus
\begin{equation*}
\frac{d}{dt}\int_{\Omega}u^2\ dx\leq C\int_{\Omega}u^2\ dx,
\end{equation*}
and, since $u(x,0)=0$, $u(x,t)=0$ for $t\leq t^*$.
\end{proof}

\begin{nota}
If we substitute the Dirichlet condition $u(x,t)=0$, by the Newman condition $\pd{u}{n}=0$, in the boundary, then the solution of this new  problem satisfies
$$\int_{\Omega} u\ dx=\int_{\Omega} u_0\ dx+\int_0^t\int_{\Omega}f\ dxdt =g(t),$$
where $g(t)$ is defined by the initial data. Using the inequality
$$\left|\int_{\Omega} u\ dx\right|\leq C(\Omega)\left(\int_{\Omega}u^2\ dx\right)^{\frac12},$$
we can prove that
$$a(u)=\left(\int_{\Omega}u^2\ dx\right)^{\gamma}\geq \left(\frac{1}{C(\Omega)}\int_{\Omega} u\ dx\right)^{2\gamma}=\left(\frac{g(t)}{C(\Omega)}\right)^{2\gamma},\text{ for }\gamma>0,$$
and
$$a(u)=\left(\int_{\Omega}u^2\ dx\right)^{\gamma}\leq \left(\frac{1}{C(\Omega)}\int_{\Omega} u\ dx\right)^{2\gamma}=\left(\frac{g(t)}{C(\Omega)}\right)^{2\gamma},\text{ for }\gamma<0.$$
Supposing that $u_0$ and $f$ are such that $g(t)>0$ for $t>0$, then the existence and uniqueness could be proved in $[0,T]$, for $T>0$.

\end{nota}

\section{Asymptotic behaviour}
The weak solutions of Problem (\ref{probi}) exhibit different behaviours for different values of $\gamma$. If $\gamma>0$, then we have a decay of the energy of the solution. If $\gamma<0$, then the solution vanishes in a finite time.
\begin{thm}
If $f=0$ and $\gamma>0$, then the weak solution $u$ of Problem (\ref{probi}) satisfies
$$\int_{\Omega} u^2\ dx\leq \frac{\int_{\Omega} u_0^2\ dx}{\left(1+2C_2\gamma \left(\int_{\Omega} u_0^2\ dx\right)^{\gamma}\ t\right)^{\frac{1}{\gamma}}},$$
where $C_2$ is the Poincaré constant.
\end{thm}
\begin{proof}
Multiplying the first equation of Problem (\ref{probi}) by $u$ and integrating in $\Omega$, we arrive at
$$\frac12\frac{d}{dt}\int_{\Omega} u^2\ dx+\left(\int_{\Omega} u^2\ dx\right)^{\gamma}\int_{\Omega}|\nabla u|^2\ dx =0.$$
By the Poincaré inequality, we prove that
$$\frac12\frac{d}{dt}\int_{\Omega} u^2\ dx+C_2\left(\int_{\Omega}u^2\ dx\right)^{\gamma+1} \leq 0.$$
Hence $y=\int_{\Omega} u^2\ dx$ satisfies the differential inequality
$$y'+2C_2y^{\gamma+1}\leq 0.$$
Solving this inequality, we obtain
$$y\leq \frac{y_0}{\left(1+2\gamma C_2y_0^{\gamma}t\right)^{\frac{1}{\gamma}}},$$
which is the desired estimate.
\end{proof}
Let us define the positive part of a function $h$ as
$$[h]_+=\left\{\begin{array}{ll}
h,&h\geq 0\\
0,&h<0
\end{array}\right..$$

\begin{thm}
If $f=0$ and $\gamma<0$, then the weak solution $u$ of Problem (\ref{probi}) satisfies
$$\int_{\Omega} u^2\ dx\leq \left[\left(\int_{\Omega} u_0^2\ dx\right)^{|\gamma|}-2|\gamma|C_2\ t\right]_+^{\frac{1}{|\gamma|}},$$
where $C_2$ is the Poincaré constant.
\end{thm}
\begin{proof}
Multiplying the first equation of Problem (\ref{probi}) by $u$ and integrating in $\Omega$, we arrive at
$$\frac12\frac{d}{dt}\int_{\Omega} u^2\ dx+\left(\int_{\Omega} u^2\ dx\right)^{\gamma}\int_{\Omega}|\nabla u|^2\ dx =0.$$
By the Poincaré inequality, we prove that
$$\frac12\frac{d}{dt}\int_{\Omega} u^2\ dx+C_2\left(\int_{\Omega}u^2\ dx\right)^{\gamma+1} \leq 0.$$
Hence $y(t)=\int_{\Omega} u^2\ dx$ satisfies the differential inequality
$$y'+2C_2y^{\gamma+1}\leq 0.$$
Solving this inequality, we obtain
$$y\leq \left[y_0^{|\gamma|}-2|\gamma|C_2\ t\right]_+^{\frac{1}{|\gamma|}},$$
which is the desired estimate.
\end{proof}

\begin{thm}\label{decay}
If
$$\|f\|_{L_2(\Omega)}\leq \frac{A}{(1+Bt)^{\frac{2\gamma+1}{2\gamma}}},$$
with
\begin{equation}\label{eqAB}
A<C_2\left(\int_{\Omega} u^2_0\ dx\right)^{\frac{2\gamma+1}{2}},\quad B=2\gamma C_2\left(\int_{\Omega} u^2_0\ dx\right)^{\gamma}-\frac{2\gamma A}{\left(\int_{\Omega} u^2_0\ dx\right)^{\frac12}}
\end{equation}
and $C_2$ the Poincaré constant, then the weak solution $u$ of Problem (\ref{probi}) with $\gamma>0$, satisfies
$$\int_{\Omega} u^2\ dx\leq \frac{\int_{\Omega} u^2_0\ dx}{(1+Bt)^{\frac{1}{\gamma}}}.$$
\end{thm}
\begin{proof}
Multiplying the first equation of Problem (\ref{probi}) by $u$ and integrating in $\Omega$, we obtain
$$\frac12\frac{d}{dt}\int_{\Omega} u^2\ dx+\left(\int_{\Omega} u^2\ dx\right)^{\gamma}\int_{\Omega}|\nabla u|^2\ dx =\int_{\Omega}fu\ dx.$$
By the Poincaré and Cauchy inequalities, we can prove that
$$\frac12\frac{d}{dt}\int_{\Omega} u^2\ dx+C_2\left(\int_{\Omega}u^2\ dx\right)^{\gamma+1} \leq \left(\int_{\Omega}f^2\ dx\right)^{\frac12}\left(\int_{\Omega}u^2\ dx\right)^{\frac12}.$$
Hence $y=\int_{\Omega} u^2\ dx$ satisfies the differential inequality
$$\frac12y'+C_2y^{\gamma+1}\leq C_3(t)y^{\frac12},$$
where $C_3(t)=\left(\int_{\Omega} f^2\ dx\right)^{\frac12}$.
Setting $z=y^{\frac12}$, we arrive at
\begin{equation}\label{odi2}z'+C_2z^{2\gamma+1}\leq C_3(t).\end{equation}
Considering $A$ and $B$ as defined in (\ref{eqAB}), the solution of the ordinary differential equation
$$w'+C_1w^{2\gamma+1} = \frac{A}{(1+Bt)^{\frac{2\gamma+1}{2\gamma}}}$$
is the function
$$w=\frac{w_0}{(1+Bt)^{\frac{1}{2\gamma}}},$$
which is an upper bound for the solutions of (\ref{odi2}).
Reverting to $y$, we then obtain the desired estimate.
\end{proof}

\begin{thm}
If
$$\|f\|_{L_2(\Omega)}\leq A\left[1-Bt\right]_+^{-\frac{2\gamma+1}{2\gamma}},$$
with
\begin{equation}\label{eqAB2}
A<C_2\left(\int_{\Omega} u^2_0\ dx\right)^{\frac{2\gamma+1}{2}},\quad B=\frac{2\gamma A}{\left(\int_{\Omega} u^2_0\ dx\right)^{\frac12}}-2\gamma C_2\left(\int_{\Omega} u^2_0\ dx\right)^{\gamma}
\end{equation}
and $C_2$ the Poincaré constant, then the weak solution $u$ of Problem (\ref{probi}), with $\gamma<0$, satisfies
$$\int_{\Omega} u^2\ dx\leq \int_{\Omega} u^2_0\ dx\left[1-Bt\right]_+^{\frac{1}{|\gamma|}}.$$
\end{thm}
\begin{proof}
Multiplying the first equation of Problem (\ref{probi}) by $u$, integrating in $\Omega$ and arguing in the same way as in the proof of Theorem \ref{decay}, we prove that $z=\left(\int_{\Omega} u^2\ dx\right)^{\frac12}$
satisfies the differential inequality
\begin{equation}\label{odi3}z'+C_2z^{2\gamma+1}\leq C_3(t),\end{equation}
where $C_3(t)=\left(\int_{\Omega} f^2\ dx\right)^{\frac12}$.
Considering $A$ and $B$ as defined in (\ref{eqAB2}) and  the ordinary differential equation
\begin{equation}\label{eq_W}
w'+C_2w^{2\gamma+1} = A\left[1-Bt\right]_+^{-\frac{2\gamma+1}{2\gamma}},
\end{equation}
it is easy to verify that
$$w=w_0[1+Bt]_+^{-\frac{1}{2\gamma}}$$
is a solution of Equation (\ref{eq_W}) and an upper bound for the solutions of (\ref{odi3}).
Reverting to $y$, the proof is concluded.
\end{proof}

\section{Conclusions}

We proved the existence of a global in time weak solution for a nonlocal degenerate parabolic problem with $\gamma\geq 0$,  and the existence and uniqueness of a local in time weak solution for the problem with $\gamma< 0$.
We also obtained conditions on $\gamma$, $f$ and $u_0$ which ensure that that the solutions decay in time or become extinct in finite time.

\section*{Acknowledgements}

This work was partially supported by the research projects:\\ OE/MAT/UI0212/2014 - financed by FEDER through the - Programa Operacional Factores de Competitividade, FCT - Funda\c{c}\~ao para a Ci\^{e}ncia e a Tecnologia, Portugal and MTM2011-26119, MICINN, Spain.

\bibliographystyle{plain}
\bibliography{AAD}
\end{document}